\documentclass[]{article}

\usepackage{amsmath}
\usepackage{amsfonts, amssymb}
\usepackage{amsthm}
\usepackage[margin=1.0in]{geometry}

\newtheorem{theorem}{Theorem}[section]
\newtheorem{lemma}[theorem]{Lemma}

\numberwithin{equation}{section}

\title{Refined Asymptotics in the Online Selection of an Increasing Subsequence}
\author{Amirlan Seksenbayev, Queen Mary University of London}

\begin{document}

\maketitle

\begin{abstract}
\noindent
Let $v_n$ be the maximum expected length of an increasing subsequence, which can be selected by an online nonanticipating policy from a random sample of size $n$.
Refining known estimates, we obtain an asymptotic expansion of $v_n$ up to a $O(1)$ term. The method we use is based on detailed analysis of the dynamic programming equation, and is also applicable to the online selection problem with observations occurring at times of a Poisson process.
\end{abstract}

\section{Introduction}
In the {\it online} increasing subsequence problem 
the objective  is
 to maximise the expected length of increasing subsequence selected by a non-anticipating policy from  a  sequence 
 of random items  $X_1, \dots,X_n$ sampled  independently from  known continuous distribution $F$.  
The online constraint requires to  accept or reject  $X_i$ at  time $i$ when the item is observed, with the decision on the item becoming immediately terminal.
Samuels and Steele  \cite{F} introduced the problem and proved that the maximum expected length $v_n$ has   asymptotics
\begin{align} \label{100}
	v_n \sim \sqrt{2n} \qquad \text{as } n \to \infty.
\end{align}
To compare,   the asymptotic expected length  of the {\it longest} increasing subsequence  is $2\sqrt{n}$, as is well-known in the context of the Ulam-Hammersley problem on random permutations \cite{Romik}. 
The difference in factors  reflects the advantage of a prophet with complete overview of the random sequence over a rational but nonclairvoyant  gambler learning the sequence and making decisions in real time.

The optimal value $v_n$ does not depend on the distribution $F$, and  as in the previous work we will further assume $F$ to be the uniform distribution on the unit interval. 
The tightest known bounds on $v_n$ are 
\begin{align} \label{2}
	\sqrt{2n} - 2 \log{n} -2 \leq v_n < \sqrt{2n}.
\end{align}
The upper bound appeared in \cite{G} in the context of a sequential knapsack problem and was generalised in \cite{ H}
for the problem with random sample size. The lower bound appeared recently in Arlotto et al  \cite{A}.
To derive (\ref{100}) Samuels and Steele \cite{F} employed a stationary policy which accepts the $i$th item each time $X_i$ exceeds the previous selection by no more than $\sqrt{2/n}$;
this policy, however, falls by $O(n^{1/4})$ below the upper bound (\ref{2}). To narrow the  gap Arlotto et al \cite{A} assessed a more involved state-dependent policy, which has the size of acceptance window  for $X_i$ both dependent
 on $i$ and the last selection so far. 
Based on extensive numerical simulation Arlotto et al \cite{A} also  suggested that the optimality gap (\ref{2}) can be further tightened.

In this paper we settle two conjectures from \cite{A} by showing that the maximum expected length has asymptotic expansion
\begin{align}
		v_n = \sqrt{2n} - \frac{1}{12} \log{n} + O(1) \ \ \ \ \text{ as } n \to \infty,
\end{align}
and that the state-dependent policy constructed in \cite{A}   is within $O(1)$ from the optimum. A similar expansion with the second term $(\log n)/6$ was obtained in the related  problem of online selection from random permutation of $n$ integers \cite{Peng}. 
The difference in logarithmic terms can be interpreted as advantage of a half-prophet, who knows the unordered sample values 
$\{X_1,\dots,X_n\}$  in advance but not the succession in which the items are revealed in the course of observation.

The discrete-time selection problem has a continuous-time counterpart, where observations occur at times of a Poisson process within given time horizon \cite{ E, G, D,  F}. 
Although   the Poisson model has an additional source of risk implied by  the unknown number of observations, its analysis is easier 
because the optimal value function depends on the current state and time only through the expected number of remaining items  exceeding the last selection.  
As stressed in \cite{B} the deep relation between fixed-$n$ and poissonised  sequential decision models is yet to be understood, and in this paper we will treat them in parallel.

\section{Selection from Poisson-paced observations}

\subsection{Setting and auxiliary results}

Let $\Pi$ be a random scatter of points in $[0,\infty) \times [0,1]$ spread
according to a unit rate planar Poisson point process. 
The event $(s,x)\in \Pi$, that $\Pi$ has an atom at $(s,x)$, is interpreted as item with value $x$ observed at time $s$.
A  sequence of atoms $(s_1,x_1),\dots, (s_k,x_k)$ is said to be increasing if $s_1<\cdots<s_k$ and $x_1<\cdots<x_k$.
We think of   the configuration of points in finite rectangle, $\Pi_{|[0,s]\times[0,1]}$, as information available to the decision maker at time $s\geq 0$.
Let $u(t)$ be  the maximum expected length of increasing sequence which can be selected from $\Pi$ within time horizon $t$ by a online policy adapted to the natural  filtration of the process $(\Pi_{|[0,s]\times[0,1]},~s\geq 0)$.   
We refer to \cite{ E, G, D,  F} for the formal definition of admissible policies in terms of an increasing sequence of stopping times.

The optimal policy belongs to the following class of {\it self-similar} policies. 
Let $\delta: {\mathbb R}_+\to[0,1]$ be a {\it threshold function} defining for every $t\geq 0$ the acceptance window 
$\delta(t)$ for a virtual observation at time $0$ in the selection
problem with horizon $t$.
Define  a policy $\tau$  recursively by the prescription: 
item $x$ observed at time $s\leq t$ is accepted if and only if 
\begin{equation}\label{delta}
0<\frac{x-z}{1-z}\leq \delta((t-s)(1-z)),
\end{equation} 
where $z$ is the biggest item chosen by $\tau$ before time $s$. In particular, the first selection by $\tau$ occurs at the time $\inf \{s\in[0,t]:(s,x)\in \Pi, x<\delta(t-s) \}$ (with the convention 
$\inf \varnothing =\infty$).

The rationale behind self-similar policies lies in the independence and symmetry properties of $\Pi$. Given that at time $s<t$ the last selected item is $z$,    
the future selections must be made from
the scatter $\Pi_{|[s,t]\times[z,1]}$, which is conditionally independent from  $\Pi_{|[0,s]\times[0,1]}$.  On the other hand, by a monotonic change of scales the scatter $\Pi_{|[s,t]\times[z,1]}$ can be transformed into a distributional 
copy of $\Pi_{|[0,(t-s)(1-z)]\times[0,1]}$,
hence starting from the state $(s,z)$ the maximum expected number of points selected after time $s$  is $u((t-s)(1-z))$. 
Scaling by  $1-z$ in (\ref{delta}) reduces  the uniform distribution on $[z,1]$  
(given the observation at time $s$ is bigger than $z$) to the uniform distribution on $[0,1]$.

We stress  that there are good suboptimal policies not in this class. 
For instance, a counterpart of the Samuels-Steele stationary policy, with selection criterion $0<x-z<\sqrt{2/t}$, yields
an increasing subsequence of expected length asymptotic to $\sqrt{2t}$, which is the best possible up to lower order terms.

The optimal value function $u$ is differentiable, increasing, concave and satisfies  the dynamic programming equation 
\begin{align} \label{3}
	u'(t) = \int_0^1 (u(t(1-x)) +1 - u(t))^{+} \mathrm{d}x
\end{align}
(where $y^+=\max(y,0)$)
with the initial condition $u(0)=0$,
see \cite{E, D}. 
 A closed form solution to (\ref{3}) is known only for $t\leq t_1$, when the optimal policy is `greedy',
that is selecting the chain of {\it records} from $\Pi$  
 (cf \cite{D} and  \cite{M}, Lemma 5.1). See \cite{D} for estimates on $u$.

Define $t_1$ as  the solution to $u(t_1)=1$. For the optimal policy $\tau^*$ the threshold function is 
$\delta^*(t)=1$ for $t\leq t_1$ and defined implicitly by the equation
$$
u(t(1-x))+1-u(t)=0~~~{\rm for~~}t>t_1.
$$

\par Our approach to the asymptotic analysis of (\ref{3}) hinges  on properties of the operator
\begin{align} \label{80}
	Jf(t):=\int_0^1 (f(t(1-x))+1-f(t))^{+} \mathrm{d}x, \ t\geq0,
\end{align}
which we consider acting on $C^{1}({\mathbb R}_+)$. It is easy to see that
\begin{enumerate} 
	\item[(i)] $Jf = J(f+c)$ for any constant $c$
	\item[(ii)] if, for some fixed $t$,  $f(s)-f(t) \leq g(s) - g(t)$ holds for $0<s\leq t$, then $Jf(t) \leq Jg(t)$.
\end{enumerate}
In terms of $J$ the optimality equation (\ref{3}) can be written as
\begin{align} \label{200}
	u'(t) = Ju(t), \ t\geq0.
\end{align}
By (i) and uniqueness, the general solution 
to (\ref{200}) is $u_c(t)=u(t) + c$, determined by the initial condition $u_c(0)=c$.

\par We will need  two elementary lemmas.
\begin{lemma} \label{2.1}
	Suppose $f\in C^{1}({\mathbb R}_+)$ satisfies $\underset{t \to \infty}\limsup f(t) = \infty$. 
Then there exists an arbitrarily large $x>0$, such that for some $t$
	\begin{itemize}
		\item[{\rm (a)}] $f(s)<f(t)=x$ for $0\leq s \leq t$,
		\item[{\rm (b)}] $f'(t)>0.$
	\end{itemize}
\end{lemma}
\begin{proof}
Let $g(t) = \underset{s\in [0,t]}\max f(s)$ be the running maximum. For $x > f(0)$ let 
\begin{align*}
l(x) = \min \{t: g(t) = x\} ,~~r(x) = \max \{t: g(t) = x\},
\end{align*}
which are well deffined because $g$ is nondecreasing and by the assumption satisfies $g(t)\to \infty$ as $t\to\infty$.
So $l(x) \leq r(x)$ and $f(r(x)) = f(l(x)) =x$. 
If neither $f'(l(x)) > 0$, nor $f'(r(x))>0$, then $g'(t) = 0$ for $l(x) \leq t \leq r(x)$. 
Now if the latter holds for all sufficiently large $x$, then $g'(t) = 0$ for all large enough $t$, 
but this is only possible if $f$ is bounded from the above, which is a contradition.
\end{proof}
The next  lemma enables one to compare solutions to (\ref{200}) with solutions of the analogous inequality.

\begin{lemma} \label{2.2}
Suppose $g\in C^1({\mathbb R}_+)$. If the function satisfies  $g'(t) > Jg(t)$ for all  sufficiently large $t$, 
then $\underset{t\geq 0}\sup (u(t)-g(t))<\infty$.  Likewise, if 
		 $g'(t) < Jg(t)$
	for all sufficiently large $t$, then $\underset{t\geq 0}\inf (u(t)-g(t))>-\infty$.
\end{lemma}
\begin{proof} 
	Suppose $\underset{t \to \infty}\limsup \left( u(t)-g(t) \right) = \infty$. 
By Lemma \ref{2.1} there exists an arbitrarily large constant $c>0$ such that for some $t>0$
and  all $0\leq s \leq t$ we have
		$u(t)-g(t) = c$ ,	$u(s) - g(s) \leq c$  and
	\begin{align} \label{1}
		u'(t) - g'(t) > 0.
	\end{align}
	Choosing $c$ large we may achieve that $t$ is large enough to satisfy $g'(t) > Jg(t)$.
However, by properties (i) and (ii) of $J$ for $g_c:=g+c$
	\begin{align*}
		u'(t) = Ju(t) \leq Jg_c(t)= Jg(t) <g'(t),
	\end{align*}
	which contradicts (\ref{1}). Thus $u(t)-g(t)$ must be bounded from the above. 
The second part of the lemma is proved by an analogous argument.
\end{proof}

\subsection{Asymptotic expansion of the optimal value function}

To obtain asymptotic expansion 
we will compare $u$ with  different test functions. In the first instance  we will  derive the well known asymptotics 
$u(t)\sim \sqrt{2t}$, $t\to\infty$.
To that end, consider $u_0(t) = \alpha_0 \sqrt{t}$ with $\alpha_0>0$. For this and other test functions  we may ignore singularities  at or near the origin,  since in the calculations to follow we assume $t$ large enough, so $u_0(t)$ for small $t$ can be modified in some way
 to agree with $u_0\in C^1({\mathbb R}_+)$. Using monotonicity we can write
\begin{align} \label{4}
	Ju_0(t) = \int_0^1 (\alpha_0 \sqrt{t(1-x)} - \alpha_0 \sqrt{t} +1)^{+} \mathrm{d}x = \int_0^{\delta_0(t)} (\alpha_0 \sqrt{t(1-x)} - \alpha_0 \sqrt{t} +1)\ \mathrm{d}x
\end{align}
where 

\begin{equation}\label{48}
\delta_0(t)=  \frac{2}{\alpha_0 \sqrt{t}}-\frac{1}{\alpha_0^2 t}
\end{equation}
is the unique solution to 
$\alpha_0 \sqrt{t(1-x)} - \alpha_0 \sqrt{t} +1 = 0$
(we remind that $t$ is large enough, in particular $t>(4\alpha_0^2)^{-1}$ to enable solution).
Although  direct integration in (\ref{4}) is easy, it is more instructive to first expand the integrand using
$$\sqrt{1-x}-1=-\frac{1}{2}x+O(\delta_0^2),~~t\to\infty$$
where $\delta_0=\delta_0(t)$ for shorthand, and the estimate $O(\delta_0^2)$ is uniform in 
$0\leq x\leq \delta_0$. Now integrating and plugging (\ref{48})
\begin{align} \label{5}
	\begin{split}
		Ju_0(t) &= \int_0^{\delta_0(t)} \left(\alpha_0 \sqrt{t} \left(-\frac{x}{2} + O(\delta_0^2)\right)+1 \right)\mathrm{d}x \\ 
		&= \delta_0-\alpha_0\sqrt{t} \ \frac{\delta_0^2}{4} + \alpha_0 \sqrt{t} \ O(\delta_0^3) = \frac{1}{\alpha_0\sqrt{t}} + O(t^{-1}),  
	\end{split}
\end{align}
On the other hand, 
\begin{equation}\label{u0dir}
u_0'(t)=\frac{\alpha_0}{2\sqrt{t}}.
\end{equation}
The right-hand sides of (\ref{5}) and (\ref{u0dir}) match for $\alpha_0 = \sqrt{2}$. Thus, for $t$ large 
enough, 
\begin{align*}
	u_0'(t) > Ju_0(t) \qquad \text{for} \ \alpha_0>\sqrt{2}, \\
 	u_0'(t) < Ju_0(t) \qquad \text{for} \ \alpha_0<\sqrt{2}. 
\end{align*}
Applying Lemma \ref{2.2} we see that  
$\underset{t\to \infty}\limsup (u(t)-\alpha_0\sqrt{t}) < \infty $ 
hence $ \underset{t\to \infty}\limsup  (u(t)/\sqrt{t}) \leq \alpha_0$
for  $\alpha_0>\sqrt{2}$. It follows that
\begin{align} \label{6}
 \limsup_{t\to \infty}  \frac{u(t)}{\sqrt{t}} \leq \sqrt{2}.
\end{align}
A parallel argument with $\alpha_0<\sqrt{2}$ yields 
\begin{align} \label{7} 
\liminf_{t \to \infty} \frac{u(t)}{\sqrt{t}} \geq \sqrt{2}. 
\end{align}
Combining (\ref{6}) and (\ref{7}) we obtain  $u(t) \sim \sqrt{2t}$ as wanted.

\par To obtain finer asymptotics we will compare $u$ with test  functions of the form
\begin{align}\label{u1} 
	u_1(t) = \sqrt{2t} + \alpha_1 \log{t},
\end{align}
with $\alpha_1\in{\mathbb R}$.
Note that 
\begin{align} \label{10}
	u_1'(t)=\frac{1}{\sqrt{2t}}+\frac{\alpha_1}{t},
\end{align}
so $u_1$ is eventually increasing regardless  of $\alpha_1$.
We have
\begin{align} \label{8}
	Ju_1(t) = \int_0^{\delta_1(t)} (u_1(t(1-x))-u_1(t)+1) \ \mathrm{d}x,
\end{align}
where $\delta_1(t)$ is the solution to

\begin{equation}\label{9}
u_1(t(1-x)) - u_1(t) + 1 =0.
\end{equation}
Similarly to (\ref{48}) we obtain the expansion
\begin{align} \label{51}
	\delta_1(t) = \sqrt{\frac{2}{t}} - \frac{4\alpha_1+1}{2t} + O(t^{-{3/2}}), \qquad \text{~~ } t\to \infty.
\end{align}

We wish to expand  $Ju_1(t)$ up to a term of order  $o(t^{-1})$. The calculation is facilitated by observing that 
the term $O(t^{-1})$ 
in (\ref{51}) can be ignored, since it only contributes $O(t^{-\frac{3}{2}})$ to $Ju_1(t)$.
Indeed, keeping $t$ as parameter, let us view the integral (\ref{8}) as a function of the upper limit
$$I(\delta):=\int_0^\delta (u_1(t(1-x))+1-u_1(t)){\rm d}x.$$
In view of (9)    $\delta_1:=\delta_1(t)$ is a stationary point of the integral.
Expanding at $\delta_1$ with remainder  we get for some $\gamma\in [0,1]$
$$I(\delta_1+\epsilon)-I(\delta_1)=I'(\delta_1)\epsilon + I''(\delta_1+\gamma\epsilon)\frac{\epsilon^2}{2}=
0 - t u_1'(t(1-(\delta_1+\gamma\epsilon))\frac{\epsilon^2}{2}.
$$
Now letting $t\to\infty$ and   $\epsilon=O(t^{-1})$ from (\ref{10}) we obtain
$$I(\delta_1(t)+\epsilon)-I(\delta_1(t)) =O(t^{-3/2}),$$
as claimed.

\par Retaining the leading term in (\ref{51}) and calculating 
 $I(\sqrt{2/t})$, 
 (\ref{8}) becomes
\begin{align} \label{13}
	Ju_1(t) = \frac{1}{\sqrt{2t}}-\frac{1}{t}\left(\alpha_1+\frac{1}{6}\right) + O(t^{-\frac{3}{2}}).
\end{align}
The right-hand sides of (\ref{10}) and (\ref{13}) match if
\begin{equation*}
	\alpha_1 = - \left(\alpha_1+\frac{1}{6}\right),
\end{equation*}
that is  for $\alpha_1 = -1/12$. For $\alpha_1 \neq -1/12$, for large $t$ the relation 
between $u_1'(t) $ and $Ju_1(t)$ has the same direction as the relation between $\alpha_1$ and $-1/12$.
Appealing to Lemma 2.2 again, we conclude that $u(t) - (\sqrt{2t}+\alpha_1 \log{t})$ 
is bounded from  above for  $\alpha_1>-1/12$ and  bounded from  below  for $\alpha_1<-1/12$.
Letting $\alpha_1$ approach $-1/12$ we obtain
$$\lim_{t\to\infty} \frac{u(t)-\sqrt{2t}}{\log{t}} =-\frac{1}{12},$$
whence the asymptotic expansion

\begin{equation}\label{as-exp2}
u(t) \sim \sqrt{2t} - \frac{1}{12} \log{t}.
\end{equation}

We need one more iteration to bound the remainder in (\ref{as-exp2}). This time we consider the test functions
\begin{align}\label{u2}
	u_2(t) = \sqrt{2t} - \frac{1}{12} \log{t} + \frac{\alpha_2}{\sqrt{t}}
\end{align}
with $\alpha_2\in{\mathbb R}$.     Solving $u_2(t(1-x)) - u_2(t) + 1 = 0$ for $x=\delta_2(t)$ we obtain regardless of 
$\alpha_2$
\begin{align} \label{52}
	\delta_2(t) = \sqrt{\frac{2}{t}} - \frac{1}{3t} + O(t^{-\frac{3}{2}}),
\end{align}
which is just (\ref{51}) with $\alpha_1=-1/12$. With account of the second term in (\ref{52}) we calculate

\begin{align} \label{17}
	Ju_2(t)=\int_0^{\delta_2(t)} (u_2(t(1-x))-u_2(t)+1) \ \mathrm{d}x=
\frac{1}{\sqrt{2t}} - \frac{1}{12t} +  \left(\frac{\alpha_2}{2}-\frac{\sqrt{2}}{144}\right) \frac{1}{t^{3/2}} + O(t^{-2}).
\end{align}
To match with 
\begin{align}
u_2'(t) = \frac{1}{\sqrt{2t}}-\frac{1}{12t}-\frac{\alpha_2}{2t^{3/2}}.
\end{align}
we choose $\alpha_2={\sqrt{2}}/{144}$, and repeating the above argument we conclude that 
$\underset{t\to\infty}\limsup |u(t)-u_2(t)|<\infty$. Absorbing the last term in (\ref{u2}) into $O(1)$ we arrive 
at the following result.

\begin{theorem} \label{49} The optimal value function has asymptotic expansion
\begin{align}\label{as-exp-u}
u(t) = \sqrt{2t} -\frac{1}{12} \log{t} + O(1), ~~~ t \to \infty.
\end{align} 
\end{theorem}

It is natural to conjecture that the $O(1)$ term in (\ref{as-exp-u}) has a limit. 
However, our method cannot capture constants since we nowhere used the initial condition $u(0)=0$.
We also believe that the described steps and further iteration yield, in fact, an asymptotic expansion of the 
{\it derivative} $u'$. See \cite{D} for non-asymptotic estimates of $u$ and its derivatives.

\subsection{A self-similar policy}

The threshold function $\hat{\delta}(t):=\min(\sqrt{2/t}, 1),~t>0,$ defines a self-similar policy via (\ref{delta}). Let $\hat{u}(t)$ be the expected length of subsequence 
selected by this policy in the problem with horizon $t\geq 0$.   A counterpart of  
(\ref{3}) is the integro-differential equation 
$$\hat{u}'(t)=\hat{J}\hat{u}(t),~~~\hat{u}(0)=0,$$
 where
\begin{equation}
\hat{J}f(t): = \int_0^{\hat{\delta}(t)} (f(t(1-x))+1 -f(t)) \mathrm{d}x.
\end{equation}
The operator $\hat{J}$ also has the shift and monotonicity properties (i), (ii), therefore 
the analogue of Lemma \ref{2.2} applies to $\hat{J}$. Comparing $\hat{u}$ with the same functions as above we arrive at 
the asymptotics
$$\hat{u}(t)=\sqrt{2t}-\frac{1}{12}\log t+O(1),$$ 
which taken together with (\ref{as-exp-u}) implies that 
$$\sup |\hat{u}(t)-u(t)|<\infty.$$

More generally, a policy with threshold function ${\delta}(t)=\min(\alpha t^{-1/2}, 1),~ \alpha>0,$ selects a subsequence with expected length asymptotic to $ 4\alpha (2+\alpha^2)^{-1}\,\sqrt{t}$, where
the maximum rate is achieved for $\alpha=\sqrt{2}$

\section{The discrete-time problem}

\subsection{Asymptotic expansion of the value function}
We turn now to the asymptotics of $v_n$, the optimal expected length in the problem with fixed sample size $n$. Arlotto et al (see \cite{B},  Corollary 9)    used concavity of  $(v_n)_{n\in {\mathbb N}}$ to show that $u(n)\leq v_n$.
This implies that the right-hand side of (\ref{as-exp-u}) is an asymptotic lower bound for $v_n$. We could not find, however, a de-poissonisation argument to construct a tough upper bound, 
hence will proceed  by analogy with the Poisson problem  via  a direct analysis of the dynamic programming equation.

For $z\in [0,1]$, 
let $v_n(z)$ be the maximum expected length of increasing subsequence which can be achieved with a policy  never selecting items smaller than $z$.
In particular, $v_n(0)=v_n$.
It is easy to see that  $v_k(z)$ (for any $n\geq k$)  is  the expected length of increasing subsequence which will be selected under the optimal policy  when $k$ items remain to be seen and    the last item selected so far 
  is $z$. In such situation the number of remaining items above $z$ has binomial distribution with mean $k(1-z)$. 
The optimality equation  is now a recursion \cite{B, A,  F} 
\begin{align}\label{22}
	v_k(z) = z  \ v_{k-1}(z) + \int_z^1 \max \left\{v_{k-1}(x)+1, v_{k-1}(z) \right\} \mathrm{d}x,~~~k\in {\mathbb N},
\end{align}
with $v_0(z) = 0$ and $v_1(z) = 1-z$.
Note that $v_k(z)+c$ also satisfies (\ref{22}) for any constant $c$.

\par  Next is an analogue of Lemma \ref{2.2} for the fixed-$n$ problem.
\begin{lemma} \label{23}
	Let $f_k: [0,1]\to {\mathbb R}_+, ~k\in {\mathbb N}$, be a sequence of continuous functions which satisfy
	\begin{align} \label{24}
		f_k(z) > z \ f_{k-1}(z) + \int_z^1 \max \left\{f_{k-1}(x)+1, f_{k-1}(z) \right\} \mathrm{d}x
	\end{align}
	provided $k(1-z)$ is large enough. Then the difference $v_k(z) - f_k(z)$ is uniformly bounded from  above for all $k$ and $z$. 
	Similarly, if 
	\begin{align} 
		f_k(z) < z \ f_{k-1}(z) + \int_z^1 \max \left\{f_{k-1}(x)+1, f_{k-1}(z) \right\} \mathrm{d}x
	\end{align}
	for $k(1-z)$ large enough, then the difference $v_k(z) - f_k(z)$ is uniformly bounded from  below for all $k$ and $z$. 
\end{lemma}
\begin{proof}
We will prove only the first part of the lemma, the second being analogous.
	Assume the contrary, i.e.  that  (\ref{24}) holds but
$$\underset{k(1-z)\to\infty}\limsup  (v_k(z) - f_k(z))=\infty.$$ 
Then for each $c>0$ there exist $k_0$ and $z_0$ such that $v_{k_0}(z_0) \geq f_{k_0}(z_0) + c$, and we choose the minimal such $k_0$. 
Since $v_{k_0}(1)=0<f_{k_0}(1)+c$,  by continuity we may choose $z_0\in [0,1)$ to have the equality
	\begin{align} \label{26}
		v_{k_0}(z_0) = f_{k_0}(z_0) + c.
	\end{align} 
Using the obvious upper estimate  $v_k(z)\leq k(1-z)$ we see that
   $k_0(1-z_0) \to \infty$ as  $c \to \infty$. 
Hence, for $c$ large  the inequality (\ref{24}) holds with $k=k_0, z=z_0$ and adding the constant to both sides we obtain
	\begin{align} \label{25}
		f_{k_0}(z_0) + c > z_0 ( f_{k_0-1}(z_0)+c) + \int_{z_0}^1 \max \left\{(f_{k_0-1}(x)+c)+1, (f_{k_0-1}(z_0)+c) \right\} \mathrm{d}x.
	\end{align}
On the other hand, from the optimality recursion and  the choice  of $k_0$ we also have
	\begin{align} \label{800}
		\begin{split}
			v_{k_0}(z_0) &=  z_0  \ v_{k_0-1}(z_0) + \int_{z_0}^1 \max \left\{v_{k_0-1}(x)+1, v_{k_0-1}(z_0) \right\} \mathrm{d}x \\ & < z_0 (f_{k_0-1}(z_0)+c) + \int_{z_0}^1 \max \left\{(f_{k_0-1}(x)+c)+1, (f_{k_0-1}(z_0)+c)\right\} 
\mathrm{d}x. 
		\end{split}
	\end{align}
	However, (\ref{25}) and (\ref{800})   cannot hold  together with  (\ref{26}), which is a contradiction.
\end{proof}

We will apply the lemma to compare $v_k(z)$ with suitable test functions. Given a sequence of functions $f_k(z):[0,1]\to{\mathbb R}_+,~k\in{\mathbb N}$, introduce operators  
$$\Delta f_{k}(z) := f_{k+1}(z) - f_{k}(z), ~~Gf_k(z) := \int_z^1 (f_{k}(x)+1 - f_{k}(z))^{+} \mathrm{d}x.$$
With this notation, the optimality  equation (\ref{22}) assumes the form
\begin{align} \label{28}
	\Delta v_k(z) = Gv_k(z).
\end{align}
By Lemma \ref{23}, if  for $k(1-z)$ large enough
$\Delta f_{k}(z) > Gf_k(z)$, then the difference $v_k(z) - f_k(z)$ is  bounded from  above uniformly in  $k$ and $z$; likewise
 if $\Delta f_{k}(z) < Gf_k(z)$ then  $v_k(z) - f_k(z)$ is bounded from  below.

To obtain the principal asymptotics 
consider the test function $v^{(0)}_k(s): = \gamma_0 \sqrt{k(1-z)}$  where $\gamma_0>0$ is a parameter.
Introducing  $\hat{k}:=k(1-z)$ and expanding for large $k$ we obtain
\begin{align} \label{29}
	\Delta v^{(0)}_k(z) = \gamma_0 \frac{1-z}{2\sqrt{\hat{k}}}+O\left( \frac{1}{k^{3/2}}\right).
\end{align} 
Observe that, unlike the Poisson problem, the expansion is  not in terms of the expected  number    
of future observations $\hat{k}$. This happens because $\Delta$ is the forward difference in the varible $k$ rather than $\hat{k}$.
Furthermore,

using the change of variable $y:={(x-z)}/{(1-z)}$, we can write the integral as
\begin{align} \label{32}
	Gv^{(0)}_k(z) = (1-z) \int_0^1 \left(\gamma_0 \sqrt{\hat{k}-\hat{k}y} - \gamma_0 \sqrt{\hat{k}} + 1\right)^+\mathrm{d}y,
= (1-z) \int_0^{h_0(\hat{k})}\left(\gamma_0 \sqrt{\hat{k}-\hat{k}y} - \gamma_0 \sqrt{\hat{k}} + 1\right)\mathrm{d}y,
\end{align}
where $h_0(\hat{k})$ is the solution to
\begin{align} \label{30}
	\gamma_0 \sqrt{\hat{k}-x} - \gamma_0 \sqrt{\hat{k}} + 1= 0.
\end{align} 
For $ \hat{k} \to \infty$ we have
\begin{align} \label{31}
	h_0(\hat{k}) = \frac{2}{\gamma_0\sqrt{\hat{k}}}+O(\hat{k}^{-1}). 
\end{align}

Expanding the integrand in  (\ref{32}) yields
\begin{align*} 
	Gv^{(0)}_k(z) = (1-z) \int_0^{h_0(\hat{k})} \left(1 -\gamma_0 \sqrt{\hat{k}} \,\frac{y}{2} + O((h_0(\hat{k})^2) \right) \mathrm{d}y,
\end{align*}
hence integrating  and using (\ref{31}) 
\begin{align} \label{34}
	Gv^{(0)}_k(z) \sim \frac{1-z}{\gamma_0\sqrt{\hat{k}}}\,, ~~~\hat{k}\to\infty.
\end{align}

The match between (\ref{29}) and (\ref{34}) occurs for $\gamma_0 = \sqrt{2}$. Therefore, 
 applying Lemma \ref{24} and mimicking the argument in Section 3.2  we conclude that
\begin{align} \label{101}
	v_k(z) \sim \sqrt{2k(1-z)}, ~~~
\end{align}
as $k(1-z) \to \infty$. This can be viewed as the maximum expected length of increasing subsequence chosen from $N$ items, with binomially distributed $N$  (see \cite{H} p. 945 and 
\cite{F} p. 1083).

\par For better  approximation we consider the test function $v^{(1)}_k(z) = \sqrt{2k(1-z)} + \gamma_1 \log{(k(1-z))}$ with $\gamma_1\in{\mathbb R}$.  The forward difference becomes

$$\Delta v^{(1)}_k(z)= \sqrt{2\hat{k}} \left(\left(1+\frac{1}{{k}}\right)^{1/2} - 1\right) + \gamma_1 \log{\left(1 + \frac{1}{{k}}\right)}.$$

Using  Taylor expansion with remainder  yields    
\begin{align}\label{35}
	\Delta v^{(1)}_k(z) = \frac{1-z}{\sqrt{2\hat{k}}} + \gamma_1\frac{1-z}{\hat{k}} + O({k}^{-{3}/{2}}), ~~~k\to\infty.
\end{align}
On the other hand,  using substitution $y ={(x-z)}/{(1-z)}$
\begin{align}\label{36}
	\begin{split}
		Gv^{(1)}_k(z) &= \int_z^{1} \left( \sqrt{2k(1-x)} - \sqrt{2k(1-z)} + \gamma_1 \log{\left(k(1-x)\right)} - \gamma_1 \log{\left(k(1-z)\right)} +1  \right)^{+} \mathrm{d}x \\ &=
(1-z) \int_0^{1} \left( \sqrt{2\hat{k}}\left((1-y)^{1/2}-1\right) + \gamma_1 \log{(1-y)} +1  \right)^{+} \mathrm{d}y\\& = (1-z) \int_0^{h_1(\hat{k})} \left( \sqrt{2\hat{k}}\left((1-y)^{1/2}-1\right) + \gamma_1 \log{\left(1-y\right)} +1  \right) \mathrm{d}y, \nonumber
	\end{split}
\end{align}
where $h_1(\hat{k})$ solves
\begin{equation*}
	\sqrt{2\hat{k}}\left((1-y)^{1/2}-1\right) + \gamma_1 \log{\left(1-y\right)} +1 = 0.
\end{equation*}
For $\hat{k}\to\infty$
\begin{align} \label{37}
	h_1(\hat{k}) = \sqrt{\frac{2}{\hat{k}}} - \left(\frac{1}{2} + 2\gamma_1 \right) \frac{1}{\hat{k}} + O(\hat{k}^{-{3}/{2}}).
\end{align}
Hence integrating and expanding 
\begin{align} \label{38}
	Gv^{(1)}_k(z) = \frac{1-z}{\sqrt{2\hat{k}}} - \left(\gamma_1 +\frac{1}{6} \right) \frac{1}{\hat{k}} + O(\hat{k}^{-{3/2}}),
\end{align}	
where actually only the first term in (\ref{37}) was needed for calculation.
Expansions (\ref{35}) and (\ref{38}) match for $\gamma_1=-\frac{1}{12}$. Thus, another application of Lemma \ref{24} gives us 
\begin{align} \label{39}
	v_k(z) \sim \sqrt{2k(1-z)} - \frac{1}{12} \log{(k(1-z))}\,,~~~k(1-z)\to\infty.
\end{align}

We need one more iteration to bound the remainder.
 Consider the test functions 
\begin{align*}
v^{(2)}_k(z) = \sqrt{2k(1-z)} - \frac{1}{12} \log{(k(1-z))} + \gamma_2 \frac{1}{\sqrt{k(1-z)}}\,,~ ~\gamma_2\in{\mathbb R}.
\end{align*}
For $k\to\infty$ we obtain the expansion for the difference
\begin{align} \label{40}
	\Delta v^{(2)}_k(z) \sim \frac{1-z}{\sqrt{2\hat{k}}} - \frac{1-z}{12\hat{k}} + \frac{1-z}{\hat{k}^{{3}/{2}}} \left(-\frac{\gamma_2}{2} - \frac{1-z}{8} \right),
\end{align}
uniformly in $z\in[0,1)$, and with some more effort for the integral
\begin{align} \label{42} 
	Gv^{(2)}_k(z) \sim \frac{1-z}{\sqrt{2\hat{k}}} - \frac{1-z}{12\hat{k}}+ \frac{1}{\hat{k}^{3/2}} \left(\frac{\gamma_2}{2} - \frac{\sqrt{2}}{144}\right),~~ \hat{k}\to\infty.
\end{align}
Since $z\in[0,1)$, we have
\begin{align} \label{45}
	-\frac{\gamma_2}{2} - \frac{1}{8} \leq -\frac{\gamma_2}{2} - \frac{1-z}{8} < -\frac{\gamma_2}{2}.
\end{align}
Appealing to  (\ref{40}), (\ref{42}) and the first inequality in  (\ref{45}), we conclude that, for large $k(1-z)$,
\begin{align*}
	\Delta v^{(2)}_k(z) > Gv^{(2)}_k(z) \ \ \ \text{for ~} \gamma_2 \leq \frac{\sqrt{2}-18}{144},
\end{align*}
hence, by Lemma \ref{23},  $v_k(z) - v^{(2)}_k(z)$ for such $\gamma_2$ is bounded from  above. 
On the other hand, exploiting the second inequality in (\ref{45}), we derive that for large $k(1-z)$
\begin{align*} 
	\Delta v^{(2)}_k(z) < Gv^{(2)}_k(z) \ \ \ \text{for ~} \gamma_2 \geq \frac{\sqrt{2}}{144}.
\end{align*}
thus  by the lemma  $v_k(z) - v^{(2)}_k(z)$ for such $\gamma_2$ is bounded from  below. 

It follows readily that
\begin{align} \label{44} 
	v_k(z) =  \sqrt{2k(1-z)} - \frac{1}{12} \log{(k(1-z))} + O(1), ~~~k(1-z) \to \infty.
\end{align}
Our main result is the special case $z=0$:
\begin{theorem}\label{90}
The maximum expected length 
	satisfies
	\begin{align*}
	v_n = \sqrt{2n} - \frac{1}{12} \log{n} + O(1), ~~n\to\infty.
	\end{align*}
\end{theorem}

Comparing with Theorem \ref{49} we see that the Poisson and fixed-$n$ problems are asymptotically similar in a very strong sense:

$$\sup |v_n-u(n)|<\infty.$$

\subsection{A state-dependent policy}

Suppose $z$ is the last selection so far and $x\in[0,1]$ the $k$th-to-last item. Standardising the variables, the acceptance criterion for the policy from Arlotto et al \cite{A} is 
\begin{equation}\label{ArlottoPol}
0<\frac{x-z}{1-z}\leq \sqrt{\frac{2}{k(1-z)}}.
\end{equation}
The analogy with self-similar policy from Section 2.3 must be obvious.

More generally, for $k\in{\mathbb N}$ let  $h_k(z):[0,1]\to[0,1]$ be threshold functions which define a policy via the acceptance criterion
$$
0<\frac{x-z}{1-z}\leq h_k(z).
$$
The corresponding value function satisfies the recursion
$$\hat{v}_{k+1}(z) - \hat{v}_k(z) = \int_z^{z+(1-z){h}_k(z)} \left(\hat{v}_{k}(x)-\hat{v}_{k}(z)+1\right)\mathrm{d}x.$$
Analysis of this equation for the policy  (\ref{ArlottoPol})   is completely analogous to that of (\ref{28}), leading to the same asymptotics as in (\ref{44})
 \begin{align} 
	\hat{v}_k(z) = \sqrt{2k(1-z)} - \frac{1}{12} \log{k(1-z)} + O(1).
\end{align}
Taken together with (\ref{44})
this settles the conjecture in \cite{A} that the policy (\ref{ArlottoPol}) is within a constant from optimality uniformly in $n$.

\bibliographystyle{amsplain}

\end{document}